\documentclass[reqno, 12pt]{amsart}

\usepackage[utf8]{inputenc}
\usepackage[T1]{fontenc}
\usepackage{lmodern}
\usepackage{amsmath}
\usepackage{amsthm}
\usepackage{amssymb}
\usepackage{enumerate}
\usepackage{mathrsfs}
\usepackage{mathscinet}
\usepackage{stmaryrd} 
\usepackage{graphicx}
\usepackage{xcolor}
\usepackage[all,cmtip]{xy}
\usepackage{url}
\usepackage{hyperref}
\usepackage[a4paper, hmargin=3cm, vmargin=3cm]{geometry}

\newcommand\mathens[1]{\mathbb{#1}} 

\newcommand{\ud}{\mathrm{d}}

\newcommand{\F}{\mathens{F}}
\newcommand{\N}{\mathens{N}}
\newcommand{\Z}{\mathens{Z}}
\newcommand{\Q}{\mathens{Q}}
\newcommand{\R}{\mathens{R}}
\newcommand{\C}{\mathens{C}}
\newcommand{\D}{\mathens{D}}
\newcommand{\T}{\mathens{T}}
\newcommand{\U}{\mathens{U}}
\newcommand{\CP}{\C\mathrm{P}}
\newcommand\sphere[1]{\mathens{S}^{#1}}
\newcommand\proj[1]{\mathens{P}(#1)}

\newcommand{\la}{\left\langle}
\newcommand{\ra}{\right\rangle}

\DeclareMathOperator{\card}{Card}

\newcommand{\ham}{\mathrm{Ham}}

\newcommand{\id}{\mathrm{id}}

\DeclareMathOperator\im{im}
\DeclareMathOperator\ind{ind}

\DeclareMathOperator\order{order}
\DeclareMathOperator\fix{Fix}

\newcommand\compi{\circledast}
\newcommand\compii{\diamond}

\newtheorem{thm}{Theorem}[section]
\newtheorem{lem}[thm]{Lemma}
\newtheorem{cor}[thm]{Corollary}
\newtheorem{prop}[thm]{Proposition}

\newtheorem{prop-def}[thm]{Definition-proposition}

\theoremstyle{definition}

\theoremstyle{remark}

\newcommand\action{\mathcal{T}}
\newcommand\lochom{\mathrm{C}}
\newcommand\betamax{\beta}
\newcommand\betatot{\beta_{\mathrm{tot}}}

\newcommand\pj{\mathrm{pj}}

\newcommand\bepsilon{{\boldsymbol{\varepsilon}}}
\newcommand\bdelta{{\boldsymbol{\delta}}}
\newcommand\bsigma{{\boldsymbol{\sigma}}}

\newcommand\bB{{\mathcal{B}}}
\newcommand\bq{{\mathbf{q}}}
\newcommand\RS{{\R^*_+\times S^1}}

\makeatletter\let\@wraptoccontribs\wraptoccontribs\makeatother

\begin{document}

\title
{On the Hofer-Zehnder conjecture on weighted projective spaces}

\author[S. Allais]{Simon Allais}
\address{Simon Allais, Université de Paris,
IMJ-PRG,
\newline\indent  8 place Aurélie de Nemours,
75013 Paris, France}
\email{simon.allais@imj-prg.fr}
\urladdr{http://perso.ens-lyon.fr/simon.allais/}

\date{January, 2021}

\begin{abstract}
    We prove an extension of the homology version of the
    Hofer-Zehnder conjecture proved by Shelukhin to the
    weighted projective spaces which are symplectic orbifolds.
    In particular, we prove that if the number of fixed points
    counted with their isotropy order as multiplicity
    of a non-degenerate Hamiltonian diffeomorphism of such a space
    is larger than the minimum number possible,
    then there are infinitely many periodic points.
\end{abstract}
\maketitle

\section{Introduction}

We are interested in the study of the Hofer-Zehnder conjecture
in a specific class of symplectic orbifolds: the weighted projective
spaces.
These spaces appear as symplectic reductions: let us fix a tuple of
weights $\bq = (q_0,\ldots,q_d)\in (\N^*)^{d+1}$ and define
the Hamiltonian map $K_\bq : \C^{d+1}\to \R$,
\begin{equation*}
    K_\bq(z) := \pi \sum_{j=0}^d q_j|z_j|^2,
\end{equation*}
this Hamiltonian induces the $S^1$-action on $\C^{d+1}$
defined by $\bar{t}\cdot (z_j) := (e^{2i\pi q_j t}z_j)$ and
preserving the weighted sphere $\sphere{}(\bq) := K_\bq^{-1}(\pi)$,
the weighted projective space with weights $\bq$ is the symplectic
orbifold $\CP(\bq) := \sphere{}(\bq)/S^1$.
The symplectic form $\omega$ of $\CP(\bq)$ is the only 2-form
that satisfies $p^*\omega = i^*\Omega$ where $p:\sphere{}(\bq)\to\CP(\bq)$
is the quotient map, $i:\sphere{}(\bq)\hookrightarrow \C^{d+1}$
is the inclusion and $\Omega := \sum_j \ud x_j\wedge\ud y_j$
is the canonical symplectic form of $\C^{d+1}$.
The study of Hamiltonian dynamics on $\CP(\bq)$ is equivalent
to the study of the Hamiltonian dynamics on $\C^{d+1}\setminus 0$ restricted
to the flows induced by positively $2$-homogeneous Hamiltonian
$(H_t)$ commuting with $K_\bq$.
We study Hamiltonian diffeomorphisms $\varphi\in\ham(\CP(\bq))$,
which are time-one flows of this dynamics.
The case where all the weights $q_j$'s equal $1$ correspond
to the dynamics on the complex projective space $\CP^d$,
it was proved by Fortune-Weinstein that the number
of fixed points of such a diffeomorphism is always
$\geq d+1$, as was conjectured by Arnol'd \cite{For85}.
We first show the following
generalization of Fortune-Weinstein theorem.
Let $|\bq| := \sum_j q_j$ and for $z\in\CP(\bq)$ let us
denote $\order(z)\in\N$ the order of the isotropy group
of any lift of $z$ to $\sphere{}(\bq)$.

\begin{thm}\label{thm:arnold}
Every Hamiltonian diffeomorphism $\varphi$ of $\CP(\bq)$ satisfies
\begin{equation*}
\sum_{x\in\fix(\varphi)} \order(x) \geq |\bq|.
\end{equation*}
\end{thm}

Surprisingly, the number of fixed points is replaced by
a weighted count of the fixed points.
Of course, when $\bq=(1,\ldots,1)$, we get the Fortune-Weinstein theorem
back.
In \cite{Lu08}, the author claims that the $d+1$ lower bound
is also satisfied by the unweighted count of the fixed points.
However, the proof contains gaps where the orders of isotropy groups
should intervene\footnote{Just before defining the family
    $\Omega_z$ in the beginning of the proof,
    the fact that $\Pi\circ\tilde{z}_1 = \Pi\circ\tilde{z}_2$
should imply $\order(z_1)(\lambda_1-\lambda_2)\in 2\pi\Z$
rather than $\lambda_1 - \lambda_2 \in 2\pi\Z$.}.
Let us remark that Theorem~\ref{thm:arnold} is obvious when
the $q_j$'s are prime with each other two by two since for every
$j\in\{ 0,\ldots,d\}$, there is only one point $z_j\in\CP(\bq)$
such that $\order(z_j)=q_j$ and every diffeomorphism
preserves the order.
Similarly, Theorem~\ref{thm:arnold} is a consequence of the
Fortune-Weinstein theorem when $\bq$ satisfies: $\forall i, j$,
either $q_i = q_j$ or $q_i$ is prime with $q_j$
(by considering the restriction of $\varphi$ to the weighted projective
subspace $\{ z\ |\ \order(z) = q_j\}$ for each $j$).

The main subject of our article is the study of periodic points of
$\varphi\in\ham(\CP(\bq))$, which are fixed points of $\varphi^k$
for some $k\in\N^*$.
On the tori $\T^{2d}$,
Conley conjectured that
every Hamiltonian diffeomorphism has infinitely many periodic points.
This statement was proven by Hingston
\cite{Hin09} after decades of advances \cite{CZ86,SZ92,FH03,LeC06}
and then generalized to a large class of symplectic manifolds
by Ginzburg \cite{Gin10}, Ginzburg-Gürel 
\cite{GG12,GG15,GG19} and Orita \cite{Ori19}.
However, the Conley conjecture does not hold in $\CP(\bq)$:
the Hamiltonian diffeomorphism
\begin{equation*}
    \left[z_0 : z_1 : \cdots : z_d \right] \mapsto
    \left[e^{2i\pi a_0}z_0 : e^{2i\pi a_1}z_1 : \cdots : e^{2i\pi
    a_d}z_d \right],
\end{equation*}
with rationally independent 
$a_0,\ldots, a_d\in\R/\Z$, have only $d+1$ periodic points:
the projection of the canonical base.
In this case, one has equality at Theorem~\ref{thm:arnold}.
Hofer-Zehnder conjectured that the only case for which a Hamiltonian
diffeomorphism of a symplectic manifold can have finitely many periodic points is
when its periodic points are fixed and in the minimal number
possible \cite[p.~263]{HZ94}.
The conjecture was inspired by a theorem of Franks
showing that every area preserving homeomorphism isotopic to identity
has 2 or infinitely many periodic points (which implies
the Hofer-Zehnder conjecture in $\CP^1$) \cite{Fra92,Fra96}.
Collier \emph{et al.} gave a proof of Franks theorem in the case of
$\ham(\CP^1)$ using symplectic tools \cite{CKRTZ}.
The higher achievement in proving this conjecture is
Shelukhin's theorem showing a homology version of this conjecture
in a class of symplectic manifolds including $\CP^d$ \cite{She19}
(see also \cite{CGG20,OnHoferZehnderGF}).
In this article, we prove an extension of this theorem to
the weighted projective spaces.
Following Shelukhin, we introduce a homology count of the fixed points
of $\varphi\in\ham(\CP(\bq))$:
\begin{equation*}
    N(\varphi;\F) := \sum_{x\in\fix(\varphi)} O(z;\F) \in \N,
\end{equation*}
where the $O(z;\F)\in\N$ are numbers linked to the local homology groups
of the fixed points $z$ and depending on a coefficient field $\F$
(see (\ref{eq:N}) for a precise definition of the homology count,
the precise definition of $N(\varphi;\F)$ also depends on
the choice of isotopy from $\id$ to $\varphi$).
In the case where the fixed point $z$ is non-degenerate
(\emph{i.e.} $1$ is not an eigenvalue of $\ud\varphi(z)$),
$O(z;\F) = \order(z)$.
Therefore, if every fixed point is non-degenerate, $N(\varphi;\F)$
is the weighted count of fixed points used in Theorem~\ref{thm:arnold}.

\begin{thm}\label{thm:main}
    Every Hamiltonian diffeomorphism $\varphi$ of $\CP(\bq)$
    such that $N(\varphi;\F) > |\bq|$ for some field $\F$ whose characteristic is
    either $0$ or prime with each $q_j$ has infinitely many
    periodic points.
    Moreover,
    when $\varphi$ has finitely many fixed points,
    if $\F$ has characteristic $0$ in the former assumption, there exists $A\in\N$ such
    that, for all prime $p\geq A$, $\varphi$ has a $p$-periodic point
    that is not a fixed point;
    if $\F$ has characteristic $p\neq 0$, $\varphi$ has infinitely many
    periodic points whose period belongs to $\{ p^k\ |\ k\in\N \}$.
\end{thm}

Using the fact that $O(z;\F)=\order(z)$, we get the following answer
to the generalized Hofer-Zehnder conjecture in the non-degenerate case.

\begin{cor}
    Every Hamiltonian diffeomorphism $\varphi$ of $\CP(\bq)$ such that
    \begin{equation*}
        \sum_{\substack{x\in\fix(\varphi)\\ x \text{ non-degenerate}}}
        \order(x) > |\bq|
    \end{equation*}
    has infinitely many periodic points.
\end{cor}

The proof of Theorem~\ref{thm:main} is an adaptation
of our proof of the theorem of Shelukhin in the case of $\CP^d$
\cite{OnHoferZehnderGF}.
It is based on the proofs given by Givental and Théret of the Fortune-Weinstein
theorem using generating functions \cite{Giv90,The98},
we mention that
these proofs can be easily adapted to show Theorem~\ref{thm:arnold}.
The technical base of our proof
is prior to Floer theory and does not appeal to
the $J$-holomorphic curve theory:
it relies on finite-dimensional critical point theory and
classical algebraic topology.
The key ideas of the proof of Theorem~\ref{thm:main} are due to Shelukhin:
we are studying a barcode that we can associate to the persistence
module $(G^{(-\infty,t)}_*(\varphi;\F))_t$ induced by
the generating functions homology of $\varphi$
(as for $N(\varphi;\F)$, it also depends on the choice of isotopy
from $\id$ to $\varphi$).
The definition of such homology groups was introduced in \cite{OnHoferZehnderGF}
in the case of $\CP^d$ and is a generating functions counterpart to the
Floer homology groups $H^{(-\infty,t)}_*(\varphi)$
inspired by previous constructions of Viterbo \cite{Vit} and Traynor
\cite{Tray} in the case of compactly supported Hamiltonian
diffeomorphisms of $\C^d$.
The theory of barcodes in symplectic topology
was introduced by Polterovich-Shelukhin in \cite{PS16}.
Adapting Shelukhin's proof, we show that
if the homology count $N(\varphi;\F)$ is greater than $|\bq|$,
the barcode of $(\varphi;\F)$ must contain finite bars.
Since each finite bar must have a length lower than $1$ whereas
the sum of the length of (representatives of) finite bars $\betatot(\varphi^k;\F)$
is diverging to $+\infty$ for good choices of powers $k$ and
fields $\F$ according to a Smith-type inequality,
the number of finite bars 
diverges (more precisely: the number of $\Z$-orbits of finite bars
diverges)
and so does the number of periodic points.

Major works in symplectic topology recently used symplectic orbifolds
\cite{CGHMSS,PS21} and
we are hoping that this study will contribute to a better understanding
of what one should expect of an orbifold Hamiltonian Floer homology theory.
Indeed, our ``weighted'' result is not the first intriguing phenomenon
observed in this topic:
a recent work extending the Floer homology to global quotient orbifolds
(\emph{i.e.} orbifolds obtained as
quotient of a manifold by a finite group, contrary to $\CP(\bq)$)
relates this homology theory to the Chen-Ruan homology \cite{Mir19}.

\subsection*{Organization of the paper}
In Section~\ref{se:preliminaries},
we discuss preliminary results needed for the construction
of the generating functions homology:
homology projective join on weighted projective spaces
and generating functions of $\RS$-equivariant Hamiltonian
diffeomorphisms.
In Section~\ref{se:GFhomology}, we extend the construction
of the generating functions homology for complex projective spaces
to weighted projective spaces.
We then study the spectral invariants associated with
these homology groups and
derive Theorem~\ref{thm:arnold} and the universal
bound on the length of finite bars
(Theorem~\ref{thm:betamax}).
In Section~\ref{se:smith}, we show the Smith-type inequality
satisfied by the length of the finite bars of the barcode
associated with a diffeomorphism $\varphi$.
In Section~\ref{se:proof}, we prove Theorem~\ref{thm:main}.
In Appendix~\ref{se:orbibundle}, we discuss an extension of
the Thom isomorphism and the Gysin long exact sequence
to orbibundles that is needed in our article.

\subsection*{Acknowledgments}
I am grateful to Lisa Traynor who asked me about the possible
extension of my proof to this singular setting.
I am also thankful to Egor Shelukhin, Marco Mazzucchelli,
Vincent Humilière and Sobhan Seyfaddini for their support.
I am especially grateful to Sheila Sandon for the fruitful
related discussions we had as I was invited by her.

\section{Preliminaries}
\label{se:preliminaries}

\subsection{The category of weighted projective spaces}

Let us fix convention and notation about weighted projective spaces.

Every spaces considered here are finite dimensional.
Let $E$ be a complex vector space and $\rho : S^1 \to GL(E)$ a smooth
group morphism defining a linear $S^1$-action of $E$.
Formally a weighted projective space will consist of the data $(E,\rho)$, where
$E$ and $\rho$ are as above, and be denoted $\proj{E,\rho}$.
The group $\RS$ acts on $E\setminus 0$ by
\begin{equation*}
(\lambda,t)\cdot z := \lambda\rho(t)z,\quad \forall (\lambda,t)\in\RS, \forall z\in E\setminus 0.
\end{equation*}
The induced orbifold $(E\setminus 0)/(\RS)$ is naturally associated with $\proj{E,\rho}$ and
we will often identify $\proj{E,\rho}$ with this space, by a slight abuse of notation.
A morphism from $\proj{E,\rho}$ to $\proj{E',\rho'}$ is a class of injective $S^1$-equivariant
linear morphisms $(E,\rho)\to (E',\rho')$ under the equivalence relationship $\sim$ defined by
\begin{equation*}
f\sim g \Leftrightarrow \exists (\lambda,t)\in\RS,  f = \lambda \rho'(t)\circ g.
\end{equation*}
Every morphism induces a natural orbifold map, we will identify morphism and induced map
by a slight abuse of notation.
A weighted projective subspace $P\subset \proj{E,\rho}$ is a projective space $P=\proj{F,\rho'}$
induced by an $S^1$-invariant subspace $F\subset E$ with $\rho'$ the natural restricted action.

Given an $S^1$-action $\rho:S^1\to GL(E)$, there exists a base $(v_0,\ldots,v_n)$
and integers $q_0,\ldots,q_n\in\Z$ such that
\begin{equation*}
\rho(t)v_j := e^{2i\pi t q_j}v_j, \quad \forall t\in S^1,\forall j\in\{
0,\ldots, n\},
\end{equation*}
seeing $S^1$ as $\R/\Z$.
The multiset $\{q_0,\ldots, q_n\}$ is uniquely defined by $(E,\rho)$ and called
the weights of $\proj{E,\rho}$; it defines a fonctor from the category of weighted projective
spaces to the category of multisets.
We will only study weighted projective spaces with positive weights.
An usual (or ``unweighted'') projective space is a weighted projective
space whose weights are all equal to one.
Given $\bq \in (\N^*)^{n+1}$, let $\rho_\bq : S^1 \to GL_{n+1}(\C)$
be such that
\begin{equation}\label{eq:rhoq}
\rho_\bq (t)\varepsilon_j := e^{2i\pi t q_j}\varepsilon_j, \quad \forall t\in
S^1,\forall j\in\{ 0,\ldots, n\},
\end{equation}
where $(\varepsilon_j)$ is the canonical base.
We denote $\CP(\bq) := \proj{\C^{d+1},\rho_\bq}$.
Every weighted projective space with weights $\bq$ is isomorphic to $\CP(\bq)$;
the category of weighted projective spaces up to isomorphism is equivalent to
the category of multisets of $\N^*$.

\subsection{Projective join}

Given two weighted projective spaces $P_j := \proj{E_j,\rho_j}$,
$j\in\{1,2\}$,
their projective join $P_1*P_2$ is the weighted projective space
$\proj{E_1\times E_2,\rho_1\times\rho_2}$.
The spaces $P_1$ and $P_2$ are naturally included in $P_1*P_2$
\emph{via} $E_1\times 0\subset E_1\times E_2$ and
$0\times E_2\subset E_1\times E_2$.
Given subsets $A_j\subset P_j$, one can also define the projective join
$A_1*A_2\subset P_1*P_2$ by $A_1\cup A_2 \cup \pi(\tilde{A}_1\times\tilde{A}_2)$,
where $\pi : (E_1\times E_2)\setminus 0\to P_1*P_2$ is the quotient map
and the $\tilde{A}_j \subset E_j$ are the inverse images of the $A_j$'s under
$E_j\setminus 0 \to P_j$.
Given points $a_j\in A_j$, the projective line $(a_1a_2)\subset P_1*P_2$
will refer to the weighted projective line $\{a_1\} * \{a_2\}$.

Given topological space or pair $X$, $H_*(X)$ and $H^*(X)$ denote
the singular homology and cohomology groups.
When we need to explicit the ring of coefficients $R$, it
will be written $H_*(X;R)$ and $H^*(X;R)$.
In \cite[Appendix~A]{OnHoferZehnderGF}, we defined a natural morphism
$\pj_* : H_*(A\times B)\to H_{*+2}(A*B)$ in the unweighted case
called the homology projective join.
Let us extend this natural map.
Given $A\subset P$ and $B\subset P'$ subsets of weighted projective spaces,
let us define
\begin{equation*}
    E_{A,B} := \{ (a,b,c) \in A\times B\times (A*B)\ |\
        c\in (ab) \},
\end{equation*}
and projection maps $p_1:E_{A,B} \to A\times B$ and
$p_2:E_{A,B} \to A*B$.
\begin{lem}\label{lem:orbibundle}
    At the topological level,
    the map $p_1$ defines a $\CP^1$-orbibundle
    (in the sense of Appendix~\ref{se:orbibundle}),
    with the natural orientation induced by the complex structure.
\end{lem}
\begin{proof}
    We refer to Appendix~\ref{se:orbibundle} for the statement of
    the triviality condition we must show.
    In order to work with coordinates, one can assume
    $A\subset \CP(\bq)$ and $B\subset\CP(\bq')$,
    $\bq\in\N^{d+1}$, $\bq'\in\N^{d'+1}$.
    To simplify notation, let us rather express $\CP(\bq)$
    as the quotient of $\C^{d+1}\setminus 0$ under
    the equivalence relation
    \begin{equation}\label{eq:equivCP}
        z \sim z' \quad \Leftrightarrow  \quad
        z_j = \lambda^{q_j} z'_j, \quad
        \exists \lambda\in\C^*,\forall j,
    \end{equation}
    and similarly for $\CP(\bq')$ and $\CP(\bq,\bq')$.
    The covering $(V_{k,l})$ of $A\times B$ is
    \begin{equation*}
        V_{k,l} := \{ ([a],[b])\in A\times B\ |\
            a_k \neq 0 \text{ and } b_l \neq 0 \}.
    \end{equation*}
    The associated sets $U_{k,l}\subset \C^{d+d'}$
    are the maximal subsets such that the map
    \begin{equation*}
        \tilde{\varphi}_{k,l} :
        (a_0,\ldots, \hat{a}_k,\ldots,a_d,b_0,\ldots,\hat{b}_l,\ldots,b_{d'})
        \mapsto
        ([a_0,\ldots,a_d],[b_0,\ldots,b_{d'}]),
    \end{equation*}
    with $a_k := 1$ and $b_l := 1$ in the right hand side
    (the symbol $\hat{a}_k$ means that the symbol $a_k$
    is erased from the sequence), are well-defined
    $U_{k,l}\to V_{k,l}$.
    Let us denote $\U_k\subset\C$ the group of the $k$-th roots of unity.
    Then $\Gamma_{k,l} := \U_{q_k}\times\U_{q'_l}$ acts linearly on $U_{k,l}$ by
    \begin{equation*}
        (\zeta,\zeta')\cdot (a,b) :=
        \left(\zeta^{q_0}a_0,\ldots,\widehat{\zeta^{q_k}a_k},\ldots,
        \zeta^{q_d}a_d,
        (\zeta')^{q'_0}b_0,\ldots,\widehat{(\zeta')^{q'_l}b_l},\ldots,
        (\zeta')^{q'_{d'}}b_{d'}\right).
    \end{equation*}
    The maps $\tilde{\varphi}_{k,l}$ induce homeomorphisms
    $\varphi_{k,l}:U_{k,l}/\Gamma_{k,l} \to V_{k,l}$.

    Let us now define the $\Gamma_{k,l}$-invariant map
    $\tilde{\chi}_{k,l} : U_{k,l} \times \CP^1 \to
    p_1^{-1}(V_{k,l})$.
    \begin{equation*}
        (a,b,[u:v]) \mapsto
        \left(\tilde{\varphi}_{k,l}(a,b),
        \left[u^{q_0}a_0,\ldots,u^{q_d}a_d,v^{q'_0}b_0,\ldots,
        v^{q'_{d'}}b_{d'}\right]\right),
    \end{equation*}
    with $a_k := 1$ and $b_l := 1$ in the right hand side.
    These maps are invariant under the following actions of the
    $\Gamma_{k,l}$'s
    \begin{equation*}
        (\zeta,\zeta')\cdot \left((a,b),[u:v]\right) :=
        \left((\zeta,\zeta')\cdot (a,b),[\zeta^{-1}u : (\zeta')^{-1}v]\right)
    \end{equation*}
    and induce homeomorphisms
    $\chi_{k,l} : (U_{k,l}\times\CP^1)/\Gamma_{k,l} \to p_1^{-1}(V_{k,l})$.
\end{proof}

Since $\CP^1\simeq\sphere{2}$, there is a natural Gysin morphism
$p_1^*:H_*(A\times B) \to H_{*+2}(E_{A,B})$ according to
Corollary~\ref{cor:Gysin}.
We can now extend the definition of $\pj_*$ to the weighted case
by setting $\pj_* := (p_2)_*\circ p_1^*$.
Let us now get the all properties stated in \cite[Appendix~A]{OnHoferZehnderGF}
back in the weighted case.

Here, let us express $\CP(\bq)$ as the quotient of $\C^{d+1}\setminus 0$
under the equivalence relation (\ref{eq:equivCP}) in order to simplify
the notation in the following definition.
According to Kawasaki \cite{Kaw73},
let us consider the map $g_\bq : \CP^d \to \CP(\bq)$,
\begin{equation*}
    g_\bq([z_0 :\cdots : z_d]) :=
    \left[z_0^{q_0}, \ldots, z_d^{q_d}\right].
\end{equation*}
Let $\U_k\subset \C$ denote the groups of the $k$-th roots of unity
and $\U_\bq := \U_{q_0}\times\cdots\times\U_{q_d}$ acting coordinate-wise
on $\CP^d$.
The map $g_\bq$ induces a homeomorphism $\CP^d/\U_\bq \simeq \CP(\bq)$
(beware that it is not an isomorphism of orbifolds).
Let us recall the following classical result of singular homology.
\begin{lem}[{\cite[IV.3.4(c)]{Bor60}}]\label{lem:Borel}
    Given a finite $G$-action on a topological space $X$,
    the morphism induced by the quotient map in homology
    (and in cohomology)
    \begin{equation*}
        H_*(X;R)^G \to H_*(X/G;R)
    \end{equation*}
    is an isomorphism when the characteristic of $R$ is either $0$
    or prime with the order of $G$.
\end{lem}
According to this lemma, when the characteristic of
$R$ is either $0$ or prime with the order of $\U_\bq$,
that is when it is prime with any of the weights $q_j$'s,
$g_\bq$ induces the isomorphism
\begin{equation}\label{iso:CPCPq}
    (g_\bq)_* : H_*(\CP^d;R) \xrightarrow{\simeq} H_*(\CP(\bq);R).
\end{equation}
More generally, for every $A\subset\CP(\bq)$, $g_\bq$
induces the isomorphism
\begin{equation*}
    (g_\bq)_* : H_*\left(g_\bq^{-1}(A);R\right)^{\U_\bq}
    \xrightarrow{\simeq} H_*(A;R).
\end{equation*}
In order to get the properties stated in \cite[Appendix~A]{OnHoferZehnderGF}
back, let us show the commutativity of the following diagram:
\begin{equation}\label{dia:equivpj}
    \begin{gathered}
        \xymatrix{
            H_*(\widetilde{A}\times\widetilde{B};R)^{\U_\bq\times\U_{\bq'}}
            \ar[d]^-{(g_\bq\times g_{\bq'})_*}_-{\simeq}
            \ar[r]^-{\pj_*}
            & H_{*+2}(\widetilde{A}*\widetilde{B};R)^{\U_\bq\times\U_{\bq'}}
            \ar[d]^-{(g_{(\bq,\bq')})_*}_-{\simeq}
            \\
            H_*(A\times B;R) \ar[r]^-{\pj_*}
            & H_{*+2}(A*B;R)
        }
    \end{gathered},
\end{equation}
where $\widetilde{A} := g_\bq^{-1}(A)$ and
$\widetilde{B} := g_{\bq'}^{-1}(B)$ and
the actions of the group $\U_\bq\times\U_{\bq'}$ on
$\widetilde{A}\times\widetilde{B}$
and $\widetilde{A}*\widetilde{B}$ are coordinate-wise.
Let us show that the top $\pj_*$ in (\ref{dia:equivpj}) is well-defined,
that is showing that $\pj_* : H_*(\widetilde{A}\times\widetilde{B}) \to
H_{*+2}(\widetilde{A}*\widetilde{B})$ is $\U_\bq\times\U_{\bq'}$-equivariant.
The group $\U_\bq\times\U_{\bq'}$ acts on $E_{\widetilde{A},\widetilde{B}}$
by restriction of its diagonal action on
$(\widetilde{A}\times\widetilde{B})\times (\widetilde{A}*\widetilde{B})$.
Both associated projection maps $p_1$ and $p_2$ are equivariant
under these actions, so the equivariance of $\pj_*$ follows.

By naturality of the properties stated in \cite[Appendix~A]{OnHoferZehnderGF}
and naturality of the homology projective joins, diagram (\ref{dia:equivpj})
implies that these properties are still verified for our extension of
the homology projective join to the unweighted case:
for instance,
the homology projective join is associative
\begin{equation*}
    \pj_*(\pj_*(\alpha\times\beta)\times\gamma) =
    \pj_*(\alpha\times\pj_*(\beta\times\gamma)),\quad
    \forall \alpha,\beta,\gamma,
\end{equation*}
and it satisfies $\pj_*([P]\times [P']) = [P * P']$ for every (disjoint)
weighted projective spaces $P$ and $P'$.

\subsection{Generating functions of $\RS$-equivariant Hamiltonian diffeomorphisms}

In this section, we recall definitions and properties already discussed
in \cite[Section~5]{periodicCPd} and \cite[Section~3.2]{OnHoferZehnderGF}
in the case of ``unweighted'' projective space
and that generalize directly to our ``weighted'' case.
Let us fix once for all the weights $\bq = (q_0,\ldots, q_d)\in(\N^*)^{d+1}$,
the $S^1$-action of $\C^{d+1}$ will always refer to the action induced by
$\rho_\bq$ defined in
(\ref{eq:rhoq}).

Given a Hamiltonian map $(h_s) : S^1\times \CP(\bq)\to \R$, let $(H_s)$ be the
Hamiltonian map of $\C^{d+1}$ that is $2$-homogeneous, $S^1$-invariant
and whose restriction to $\sphere{}(\bq)$ lifts $(h_s)$.
Let $(\Phi_s)$ be the associated Hamiltonian flow.
An $\RS$-equivariant Hamiltonian diffeomorphism will refer to the time-one map
of such a flow.
These are smooth diffeomorphisms of $\C^{d+1}\setminus 0$ that are $S^1$-equivariant
and positively homogeneous and extends to homeomorphisms of $\C^{d+1}$.
When the restriction of such a diffeomorphism $\sigma$ to $\sphere{}(\bq)$ is $C^1$-close
to the identity, there exists a unique map $f:\C^{d+1}\to\R$ such that $f(0)=0$ and
\begin{equation*}
\forall z\in \C^{d+1}, \exists! w \in \C^{d+1}, \quad
w = \frac{z + \sigma(z)}{2} \quad \text{and} \quad 
\nabla f(w) = i(z-\sigma(z)),
\end{equation*}
where $\nabla f$ denotes the gradient of $f$
(the existence of $g:=\nabla f$ is a consequence of the implicit fonction theorem
and $g$ is a gradient because it is the graph of a Lagrangian submanifold).
The map $f$ is called the elementary generating function of $\sigma$.
It is smooth away from $0$ where it is only $C^1$, it is $S^1$-invariant and
positively 2-homogeneous (this is a consequence of the definition of $f$ and the
equivariance of $\sigma$).
In general, an $\RS$-equivariant Hamiltonian diffeomorphism $\Phi$ can be written as
$\Phi = \sigma_n\circ \cdots \circ\sigma_1 $ where every $\RS$-equivariant Hamiltonian
diffeomorphism $\sigma_j$ is sufficiently small so that it admits an elementary generating
function $f_j$.
For all $n\in\N^*$,
we will say that the $n$-tuple $\boldsymbol{\sigma}=(\sigma_1,\dotsc,\sigma_n)$
is associated with the Hamiltonian flow $(\Phi_s)$ if
there exist real numbers $0=t_0\leq t_1\leq \cdots\leq t_n = 1$
such that $\sigma_k = \Phi_{t_k}\circ\Phi_{t_{k-1}}^{-1}$.
For all $k\in\N$, we denote $\bepsilon^k$ the $k$-tuple 
\begin{equation*}
    \bepsilon^k := (\id,\ldots,\id).
\end{equation*}
More generally given a tuple $\bsigma$ or $\bq$ and an integer $n\in\N$,
$\bsigma^n$ or $\bq^n$ denotes the $n$-fold concatenation.
A continuous family of such tuples $(\bsigma_s)$ will denote
a family of tuples of the same size $n\geq 1$,
$\boldsymbol{\sigma}_s =: (\sigma_{1,s},\dotsc,\sigma_{n,s})$ such that
the maps $s\mapsto \sigma_{k,s}$ are $C^1$-continuous.
We denote by $F_{\boldsymbol{\sigma}}$ the following function $(\C^{d+1})^n\to\R$:
\begin{equation*}
    F_{\boldsymbol{\sigma}} (v_1,\dotsc,v_n) := \sum_{k=1}^{n}
        f_k\left(\frac{v_k + v_{k+1}}{2}\right) +
    \frac{1}{2}\la v_k,iv_{k+1}\ra,
\end{equation*}
with convention $v_{n+1} = v_1$, each $f_k:\C^{d+1}\to\R$ being the elementary
generating function associated with $\sigma_k$.
When $n$ is odd, $F_{\boldsymbol{\sigma}}$ is a generating function of $\sigma_n\circ\cdots\circ\sigma_1$.
Therefore, every $\RS$-equivariant Hamiltonian diffeomorphism admits a generating function.
Generating functions are $S^1$-invariant and positively 2-homogenous.
$\RS$-orbits (for the diagonal action) of
critical points of a generating function of $\Phi$ are in bijection with
$\RS$-orbits of fixed points of $\Phi$
through the map $(v_1,\ldots,v_n)\mapsto v_1$.

Given generating functions $F:\C^{d+1}\times\C^k\to\R$
and $G:\C^{d+1}\times\C^l\to\R$ of $\Phi$ and $\Psi$ respectively,
the fiberwise sum of $F$ and $G$ denotes the map
\begin{equation}\label{eq:fiberwise}
    (F+G)(x;\xi,\eta) := F(x;\xi) + G(x;\eta).
\end{equation}
Although this is not a generating function of $\Phi\circ\Psi$,
the critical points of $F+G$ are also in bijection with the fixed
points of $\Phi\circ\Psi$ \emph{via} $(x;\xi,\eta)\mapsto
x-i\partial_xG(x;\eta)/2$.

When small equivariant diffeomorphisms $\sigma_j$'s are linear, the associated
$f_j$'s are quadratic forms and so is the resulting $F_\bsigma$.
When $\sigma_n\circ\cdots\circ\sigma_1$ also admits an elementary (quadratic) generating function,
we have the following unicity lemma whose proof follows the one of
\cite[Prop.~35]{ThePHD}.
\begin{lem}
    \label{lem:unicityeqgf}
    Let $Q:\C^{d+1}\times\C^k \to \R$ be a quadratic generating function
    generating the same linear Hamiltonian diffeomorphism as the elementary
    generating function $q:\C^{d+1} \to \R$.
    Then, there exists a linear fibered isomorphism
    $A$ of $\C^{d+1}\times\C^k$ which is isotopic to the identity through
    linear fiberwise isomorphism such that
    $Q\circ A = q\oplus R$ for some quadratic form $R:\C^k\to\R$.
    More precisely, if $Q(z)=\la \widetilde{Q}z,z\ra$ with
    \begin{equation*}
        \widetilde{Q}=\begin{bmatrix} a & b\\ {}^t b & c \end{bmatrix},
    \end{equation*}
    then $c$ is invertible and
    $A(x;\xi) := (x;\xi - c^{-1}{}^t bx)$ so that
    $Q\circ A(x;\xi) = q(x) +  {}^t\xi c \xi$.
\end{lem}

\section{Generating functions homology}
\label{se:GFhomology}

In this section, we define the generating functions homology of
a Hamiltonian diffeomorphism $\varphi$ in $\CP(\bq)$
and give its main properties.
The constructions and proofs are very close to the ones of
the ``unweighted'' case $\CP^d$, so we will mainly refer
to \cite{OnHoferZehnderGF} and emphasize on the key changes.

\subsection{Action and generating functions}
\label{se:action}
Let $(\Phi_s)$ be the $\RS$-equivariant Hamiltonian flow
lifting a Hamiltonian flow $(\varphi_s)$ of $\CP(\bq)$
generated by the Hamiltonian map $(h_s)$.
Let $x\in\CP(\bq)$ be a fixed point of $\varphi:=\varphi_1$
and $u:\D^2 \to \CP(\bq)$ be an orbifold map such that
$u|_{\partial\D^2}$ corresponds to $s\mapsto \varphi_s(u)$
(a capping of $x$).
The action of the capped fixed point $(x,u)$ is the real number
\begin{equation*}\label{eq:action}
    a(x,u) := -\frac{1}{\pi}\left(
        \int_{\D^2} u^*\omega + \int_0^1 h_s\circ\varphi_s(x)\ud s
    \right).
\end{equation*}
In \cite[Proposition~5.8]{The98} gives a characterisation of
the action values of the cappings of $x$ in term of the lifted
dynamics that directly extends to our weighted case.
Let $z\in\sphere{}(\bq)$ be a lift of $x$, it is not
necessarily a fixed point of $\Phi:=\Phi_1$ but there
exist real numbers $t\in\R$ such that
$\rho_\bq (-t)\Phi(z) = z$.
Following Théret, we see that such $t$ corresponds exactly to
the action values of the cappings of $x$.
With this characterisation, it is clear that the sets of action values
of $x$ is invariant under $\Z$-translations.
Here is the major difference between the weighted and unweighted case:
the set of action value of $x$ equals $t_0+\frac{1}{\order(x)}\Z$
so the number of action values inside $[0,1)$ depends
on $\order(x)$.
This is ultimately the reason why our version of the Fortune-Weinstein theorem,
which follows Givental and Théret's steps, give multiplicity to fixed points.

In order to study the fixed points of $\varphi$,
we define continuous families of generating functions $F_t$
associated with $\rho_\bq(-t)\Phi$ for compact intervals $I$ of $t$'s.
A fixed point of action $t\in I$ corresponds to an $\RS$-orbit
of critical points of such an $F_t$.
Given a positively $2$-homogeneous map
$F$ that is invariant under
$\rho_{\bq'}$ (in our case $\bq'$ will often be 
a concatenation $\bq^n$),
we define its projectivization $\widehat{F}:\CP(\bq')\to\R$
by factoring the restriction of $F$ to $\sphere{}(\bq')$
under $\sphere{}(\bq')\to\CP(\bq')$.
Fixed points with action $t\in I$ now correspond to
critical points of $\widehat{F}_t$ with value $0$.

Let us now recall the precise construction of the families
of generating functions $F_t = F_{\bsigma_{m,t}}$.
For $m\in\N$, we define continuous tuples $t\mapsto \bsigma_{m,t}$
associated with $\rho_\bq(-t)\Phi$ for $t\in [-m,m]$ in the following
way.
Let $(\delta_t)$ be the family of small $\RS$-equivariant diffeomorphisms
$\delta_t(z):=\rho_\bq(-t)z$, $|t|<1/(2\max q_j)$.
The associated elementary generating function is
\begin{equation*}
w\mapsto -\sum_j \tan(q_j \pi t)|w_j|^2
\end{equation*}
(this function is an elementary generating function of $\delta_t$
as soon as it is well defined for the fixed $t$,
we will use it for a fixed $t$ larger than $1/(2\max q_j)$
in the proof of Proposition~\ref{prop:paction}).
Let us fix once for all an even number $n_0\geq 4\max q_j$ and
let $(\bdelta_t^{(1)})$ be the family of
$n_0$-tuples $(\delta_{t/n_0},\ldots,\delta_{t/n_0})$
generating $z\mapsto \rho_\bq(-t)z$ for $t\in (-2,2)$.
For all $m\in\N^*$, let $(\bdelta_t^{(m)})$ be a family
of $mn_0$-tuples generating $z\mapsto \rho_\bq(-t)z$ for
$t\in(-m-1,m+1)$ and satisfying 
\begin{equation}\label{eq:bdeltam}
    \bdelta_t^{(m+1)} = \left(\bdelta_t^{(m)},\bepsilon^{n_0}\right),\quad
    \forall t\in [-m,m].
\end{equation}
More precisely, let $\chi:\R\to\R$ be an odd smooth non-decreasing map
such that $\chi_m\equiv\id$ on $[-m-1/4,m+1/4]$ and $\chi_m\equiv m+1/2$ on
$[m+3/4,+\infty)$. We set
\begin{equation*}
    \bdelta_t^{(m+1)} = \left(\bdelta_{\chi_m(t)}^{(m)}
    ,\bdelta^{(1)}_{t-\chi_m(t)}\right),\quad
    \forall t\in (-m-2,m+2).
\end{equation*}

Finally, we can set
\begin{equation*}
    \bsigma_{m,t} := \left(\bsigma,\bdelta_t^{(m)}\right),\quad
    \forall t\in [-m,m].
\end{equation*}
Since $\tan$ is increasing on $(-\pi/2,\pi/2)$,
we deduce that $\partial_t F_{\bsigma_{m,t}} \leq 0$
by a straightforward computation (here it is crucial that each weight $q_j$
is positive).

\subsection{Homology of sublevel sets and local homology
of a fixed point}
\label{se:homologysublevel}

Here and throughout this paper, $H_*(X)$ and $H^*(X)$ denote respectively
the singular homology and the singular cohomology of a topological
space or pair $X$ over an indeterminate ring $R$ whose characteristic
is $0$ and prime with any of the weights $q_j$ (that have been fixed
once for all).
If one needs to specify the ring $R$, one writes $H_*(X;R)$
and $H^*(X;R)$ instead.
The following notation naturally extends the one used in the unweighted case.
Let $\bsigma$ be an $n$-tuple of small $\RS$-equivariant Hamiltonian
diffeomorphisms.
We denote by $Z(\bsigma)\subset \CP(\bq^n)$ the sublevel set
\begin{equation*}
    Z(\bsigma) := \left\{ \widehat{F}_{\bsigma} \leq 0 \right\}.
\end{equation*}
We denote by $HZ_*(\bsigma)$ the shifted homology group
\begin{equation*}
    HZ_*(\bsigma) := H_{*+(n-1)(d+1)}(Z(\bsigma)),
\end{equation*}
and if $Z(\bsigma')\subset Z(\bsigma)$,
with $\bsigma'$ an $n$-tuple, we set
\begin{equation*}
    HZ_*(\bsigma,\bsigma') := H_{*+(n-1)(d+1)}(Z(\bsigma),Z(\bsigma')).
\end{equation*}

For $m\in\N^*$ and $a\leq b$ in $[-m,m]$,
one has $F_{\bsigma_{m,b}} \leq F_{\bsigma_{m,a}}$ so
$Z(\bsigma_{m,a})\subset Z(\bsigma_{m,b})$ and
we can set
\begin{equation*}
    G_*^{(a,b)}(\bsigma,m) :=
    HZ_*(\bsigma_{m,b},\bsigma_{m,a}),
\end{equation*}
when $a$ and $b$ are not action values of $\bsigma$.
We define in the same way the cohomology analogues of these notations,
\emph{e.g}
\begin{equation*}
    G^*_{(a,b)}(\bsigma,m) :=
    HZ^*(\bsigma_{m,b},\bsigma_{m,a}) =
    H^{*+(n-1)(d+1)}(Z(\bsigma_{m,b}),Z(\bsigma_{m,a})).
\end{equation*}

This homology group can be naturally identified to
the homology of sublevel sets
of a map (see \cite[Section~5.4]{periodicCPd}):
\begin{equation*}
    G_*^{(a,b)}(\bsigma,m) \simeq H_{*+(n-1)(d+1)}
    \left(\left\{\action \leq b\right\},\left\{\action\leq a\right\}\right),
\end{equation*}
for some $C^1$-map $\action:M\to\R$ that is smooth in the neighborhood
of its critical points.
The function $\action$ is some kind of finite-dimensional action:
critical points of $\action$ are in one-to-one correspondence with
capped fixed points of $\varphi$ with action value inside $[-m,m]$.
In the unweighted case at least,
this correspondence sends critical value to action value and
Morse index up to a $(n-1)(d+1)$ shift in degree to
the Conley-Zehnder index.
More generally, the local homology of $\action$ (up to the same shift
in degree) is isomorphic to the local Floer homology of the corresponding
capped orbit (in the unweighted case at least).
Let us denote by $\lochom_*(f;x)$ the local homology
of the critical point $x$ of a map $f$:
\begin{equation*}
\lochom_*(f;x) := H_*(\{ f \leq f(x) \},\{ f \leq f(x) \}\setminus x).
\end{equation*}
We can define up to isomorphism
\begin{equation*}
    \lochom_*(\bsigma;z,t) \simeq \lochom_*\left(\widehat{F}_{\bsigma_{m,t}};\zeta\right)
    \simeq \lochom_{*+(n-1)(d+1)}\left(\action;(\zeta,t)\right),
\end{equation*}
where $\zeta\in\CP(\bq^n)$ is the critical point of
$\widehat{F}_{\bsigma_{m,t}}$ associated with the fixed point $z\in\CP(\bq)$
of action $t\in[-m,m]$ (see \cite[Section~5.5 and 5.7]{periodicCPd}
for details).
The independence on $m$ of this definition can also easily be deduced from
the isomorphism induced by $\theta_m^{m+1}$ (defined later in Section~\ref{se:direct})
on the local homologies, similarly to the unweighted case.
Local homologies $\lochom_*(\bsigma;z,t)$ and $\lochom_*(\bsigma;z,t+1)$ are isomorphic
up to a $2|\bq|$ shift in degree by the local version of the periodicity isomorphism
defined at (\ref{iso:periodic}), similarly to the unweighted case.
However, when $\order(z)\neq 1$, it is not clear whether local homology groups
associated with action values that does not differ by an integer are isomorphic (up
to a shift in degree).
For these reasons, when the grading is irrelevant, we will only specify
the action value up to an integer.

We can now define precisely $N(\bsigma;\F)$ for a choice of tuple $\bsigma$
and of field $\F$ (whose characteristic is either $0$ or prime with any of the weights)
by
\begin{equation}\label{eq:N}
    N(\bsigma;\F) := \sum_{z\in\mathrm{Fix}(\varphi)}
    \sum_{j=1}^{\order(z)}
    \dim \lochom_*(\bsigma;z,t_j(z);\F) \in \N,
\end{equation}
where $(t_j(z))$ is the increasing sequence of action values that
$z$ takes inside $[0,1)$.
The integers $O(z;\F)$ defined in the introduction for fixed points
$z\in\fix(\varphi)$ is then
\begin{equation*}
    O(z;\F) := 
    \sum_{j=1}^{\order(z)}
    \dim \lochom_*(\bsigma;z,t_j(z);\F).
\end{equation*}

We recall that an integer $k\in\N^*$ is said to be admissible for $\varphi$
at a fixed point $z$ if $\lambda^k\neq 1$ for all eigenvalues $\lambda\neq 1$
of $\ud\varphi(z)$.
Until the end of the section, $\varphi$ is associated with a tuple $\bsigma$
and the periodic points of $\varphi$ are isolated in order to simplify
the statements. The proofs do not differ from the unweighted case
(\cite[Proposition~3.1 and Corollary~3.2]{OnHoferZehnderGF}).

\begin{prop}\label{prop:gromollmeyer}
    Let $k\in\N^*$ be an admissible iteration of $\varphi$ at
    the fixed point $z$. Then as graded modules over a ring whose characteristic
    is prime with any of the weights,
    \begin{equation*}
        \lochom_*(\bsigma^k;z) \simeq \lochom_{*-i_k}(\bsigma;z),
    \end{equation*}
    for some shift in degree $i_k\in\Z$.
\end{prop}

\begin{cor}\label{cor:gromollmeyer}
    For every fixed point $z$ of $\varphi$, there exists $B>0$
    such that, for all prime $p$
    \begin{equation*}
        \dim\lochom_*(\bsigma^p;z,t_j(z);\F_p) < B, \quad
        \forall j\in \{ 1,\ldots, \order(z) \}.
    \end{equation*}
\end{cor}

\subsection{Composition morphisms and the direct system of
$G_*^{(a,b)}(\bsigma)$}
\label{se:direct}

In \cite[Section~4.1]{OnHoferZehnderGF}, we observed that the linear embedding,
\begin{equation*}
    \widetilde{B}_{n,m}(\mathbf{w},\mathbf{w}') :=
    \left(\mathbf{w},\sum_{k=1}^m (-1)^{k+1} w'_k ,\mathbf{w}'\right),\quad
    \forall \mathbf{w}\in(\C^{d+1})^n, \forall \mathbf{w}'\in(\C^{d+1})^m,
\end{equation*}
expressed in coordinates $w_k := \frac{v_k + v_{k+1}}{2}$,
satisfies for all $n$-tuples $\bsigma$ and $m$-tuples $\bsigma'$
\begin{equation*}
    F_{(\bsigma,\bepsilon,\bsigma')}
    \left(\widetilde{B}_{n,m}(\mathbf{w},\mathbf{w}')\right)
    = F_\bsigma (\mathbf{w}) + 
    F_{\bsigma'} (\mathbf{w}').
\end{equation*}
Therefore, the induced map $B_{n,m}:\CP(\bq)*\CP(\bq')
\hookrightarrow\CP(\bq,1,\bq')$
defines a natural morphism
$H_*(Z(\bsigma)*Z(\bsigma'))\to H_*(Z(\bsigma,\bepsilon,\bsigma'))$
and the composition with the homology projective join
$\alpha\otimes\beta\mapsto \pj_*(\alpha\times\beta)$
define the composition morphism
\begin{equation*}
    HZ_*(\bsigma)\otimes HZ_*(\bsigma') \to
    HZ_{*-2d}((\bsigma,\bepsilon,\bsigma'))
\end{equation*}
denoted by $\alpha\otimes\beta\mapsto \alpha\compi\beta$.
For the same formal reasons as the unweighted case,
it admits relative versions and it is associative
\cite[Section~4.1]{OnHoferZehnderGF}.

As for the unweighted case,
for a fixed $m\in\N$, the long exact sequence of triple induces
inclusion and boundary morphisms fitting into a long exact sequence:
\begin{equation*}
    \cdots \xrightarrow{\partial_{*+1}}
    G_*^{(a,b)}(\bsigma,m) \to G_*^{(a,c)}(\bsigma,m) \to G_*^{(b,c)}(\bsigma,m)
    \xrightarrow{\partial_*} G_{*-1}^{(a,b)}(\bsigma,m) \to \cdots
\end{equation*}
with $-m\leq a\leq b\leq c\leq m$ and $a,b,c$ not action values of $\bsigma$.
Using the composition morphism $\compi$, one can define
canonical isomorphisms
\begin{equation}\label{iso:Gm}
    \theta_{m}^{m+1}:G_*^{(a,b)}(\bsigma,m) \to G_*^{(a,b)}(\bsigma,m+1),
\end{equation}
for $-m\leq a\leq b\leq m$, that commutes with the above mentioned inclusion 
and boundary morphisms.
One can then define $G^{(a,b)}_*(\bsigma)$ as the direct limit of
the direct system induced by $(\theta_m^{m+1})_m$:
\begin{equation*}
    G_*^{(a,b)}(\bsigma) := \varinjlim G_*^{(a,b)}(\bsigma,m).
\end{equation*}
We then have inclusion maps
\begin{equation*}
    \cdots \xrightarrow{\partial_{*+1}} 
    G_*^{(a,b)}(\bsigma) \to G_*^{(a,c)}(\bsigma) \to G_*^{(b,c)}(\bsigma)
    \xrightarrow{\partial_*} G_{*-1}^{(a,b)}(\bsigma) \to \cdots
\end{equation*}
for all $a\leq b\leq c$ that are not action values; one can thus set
\begin{equation*}
    G_*^{(-\infty,b)}(\bsigma) := \varprojlim G_*^{(a,b)}(\bsigma),
    \quad a\to -\infty,
\end{equation*}
and one can then define $G_*^{(-\infty,+\infty)}(\bsigma)$ by taking a direct
limit in
a similar way.
The definition of (\ref{iso:Gm})
is the natural extension of the unweighted case,
let us make it explicit.
For an odd $n$,
the space $Z(\bepsilon^n)$ retracts on the projectivization of the
maximal non-positive linear subspace
of $F_{\bepsilon^n}$ which has the same homology
as a $\CP^{N-1}$ with $N=(d+1)(n+1)/2$.
Therefore,
\begin{equation*}
    HZ_*(\bepsilon^n) =
    \bigoplus_{k=-(d+1)(n-1)/2}^d R a_k^{(n)}
    \simeq H_{*+(n-1)(d+1)}\left(\CP^{(d+1)(n+1)/2-1}\right),
\end{equation*}
where $a_k^{(n)}$ is the generator of degree $2k$ identified with
the class $[\CP^l]$ of appropriate degree $2l=2k+(n-1)(d+1)$ under the isomorphism
induced by the inclusion of a maximal complex projective subspace
of $Z(\bepsilon^n)$ and (\ref{iso:CPCPq}).
We now define (\ref{iso:Gm}) by
\begin{equation*}
    \theta_m^{m+1}(\alpha) := \alpha\compi a_d^{(n_0-1)}
    \in G_*^{(a,b)}(\bsigma,m+1),
    \quad
    \forall \alpha\in G_*^{(a,b)}(\bsigma,m).
\end{equation*}
With the same proof as in the unweighted case,
one shows that $\theta_m^{m+1}$ is an isomorphism.

In \cite[Section~4.4]{OnHoferZehnderGF}, we defined a second composition
morphism $\compii$
whose goal
(only partially reached)
was to imitate the composition morphism of the
Hamiltonian Floer homology.
Let us fix 2 tuples $\bsigma$, $\bsigma'$ of odd respective sizes $n$ and $n'$,
$a,b,c\in\R$ that are not action values of $\bsigma$ and $\bsigma'$ respectively.
For sufficiently large $m,m'\in\N$, this composition morphism
$\alpha\otimes\beta\mapsto \alpha\compii\beta$,
\begin{equation*}
    HZ_*(\bsigma'_{m',c})\otimes G_*^{(a,b)}(\bsigma,m) \to
    G_{*-2d}^{(a+c,b+c)}((\bsigma',\bepsilon,\bsigma),m+m'),
\end{equation*}
naturally generalizes with the same construction.
By the same formal arguments as in the unweighted case,
it is associative and it commutes with the morphisms $\theta_m^{m+1}$
ultimately defining
\begin{equation}\label{eq:compii}
    HZ_*(\bsigma'_{m',c})\otimes G_*^{(a,b)}(\bsigma) \to
    G_{*-2d}^{(a+c,b+c)}((\bsigma',\bepsilon,\bsigma)).
\end{equation}

\subsection{Properties of the generating functions homology}
\label{se:propertiesG}

Let us first focus on the special case $\bsigma = \bepsilon$,
\emph{i.e.} $\varphi_s\equiv\id$.
Let us denote by $T_{m,t}$
the family of generating functions associated with $(\bepsilon_{m,t})_t$.
Since the elementary generating function of $\delta_s$ is a quadratic form,
so is the map $T_{m,t}$.
Since $T_{m,0}$ is a generating function of the identity, its kernel as a quadratic form
has dimension $2(d+1)$ and $\ind T_{m,0} = mn_0(d+1)$
\cite[Proposition~4.1]{OnHoferZehnderGF}.
The variation of index is governed by the Maslov index of
\begin{equation*}
t\mapsto \rho_\bq(-t) = \bigoplus_{j=0}^d e^{-2i\pi q_j t}
\end{equation*}
so that
\begin{equation*}
    \ind T_{m,t} - \ind T_{m,0} = 2\sum_{j=0}^d \lfloor q_j t\rfloor,
\end{equation*}
(See \cite[Section~3 and Lemma~5.5]{periodicCPd}).
Similarly to the unweighted case, we deduce that
the persistence module $(H_*(Z(\bepsilon_{m,t})))$ is isomorphic
to the persistence module $(H_*(\CP^{N(t)}))$, $-m< t < m$,
induced by the family of non-decreasing
projective subspaces of complex dimension
$N(t) :=  m(d+1)n_0/2 + \sum_j \lfloor q_j t \rfloor$.
We recall that the coefficient ring $R$ has a characteristic $0$ or prime
with the $q_j$'s (see (\ref{iso:CPCPq}), otherwise the family $(\CP^{N(t)})$ of
non-decreasing
unweighted projective subspaces
must be replace by a family of non-decreasing \emph{weighted} projective
subspaces).
Thus, as a graded $R$-module,
\begin{equation*}
    HZ_*(\bepsilon_{m,t}) = \bigoplus_{k=-(d+1)mn_0/2}^{d+\sum_j \lfloor q_j t\rfloor}
    Ra_k^{(mn_0+1)}(t),
\end{equation*}
where $a_k^{(mn_0+1)}(t)$ is the generator of degree $2k$ identified with
the class $[\CP^l]$ of appropriate degree $2l=2k+(d+1)mn_0$ under the previous
persistence modules isomorphism.
The inclusion morphism $HZ_*(\bepsilon_{m,t}) \to HZ_*(\bepsilon_{m,s})$
maps each $a_k^{(mn_0+1)}(t)$ to $a_k^{(mn_0+1)}(s)$
(for $-m\leq t\leq s\leq m$).
Hence,
\begin{equation*}
    G_*^{(a,b)}(\bepsilon,m) =
    \bigoplus_{k=d+\sum_j \lfloor q_j a\rfloor}^{d+\sum_j \lfloor q_j b\rfloor}
    R\alpha_k^{(m)}(a,b),
\end{equation*}
for $-m < a\leq b< m$,
where $\alpha_k^{(m)}(a,b)$ is the image of $a_k^{(mn_0+1)}(b)$
under the inclusion morphism $HZ_*(\bepsilon_{m,b})
\to G_*^{(a,b)}(\bepsilon,m)$.
Similarly to the unweighted case,
one has $\theta_m^{m+1}\alpha_k^{(m)}(a,b) = \alpha_k^{(m+1)}(a,b)$.
We set $\alpha_k(a,b) := \theta_m^\infty \alpha_k^{(m)}(a,b)$.
For $a<b<c$, if $\alpha_k(b,c)$ is well-defined,
then $\alpha_k(a,c)$ is also well-defined and sent to
the former;
there exists a well-defined $\alpha_k(-\infty,c)
\in G_{2k}^{(-\infty,c)}(\bepsilon)$ sent to $\alpha_k(a,c)$
for all $a\leq c$.
Let $\alpha_k$ be the image of $\alpha_k(-\infty,c)$ under
$G_{2k}^{(-\infty,c)}(\bepsilon) \to G_{2k}^{(-\infty,+\infty)}(\bepsilon)$, then
\begin{equation*}
    G_*^{(-\infty,+\infty)}(\bepsilon) = \bigoplus_{k\in\Z} R\alpha_k.
\end{equation*}

The ``periodicity'' isomorphism naturally extends to the weighted case as
\begin{equation}\label{iso:periodic}
    G_*^{(a,b)}(\bsigma) \xrightarrow{\simeq}
    G_{*+2|\bq|}^{(a+1,b+1)}(\bsigma),
\end{equation}
(we recall that $|\bq| := \sum_j q_j$).
Similarly to the unweighted case, it is defined using composition morphisms $\compii$ (\ref{eq:compii}),
with slight changes in the degree of the generators $a_k$ involved, let us
precise them.
Let us set
$a_d := a_d^{(mn_0+1)}(0)\in HZ_{2d}(\bepsilon_{m,0})$
and $a_{d+|\bq|} := a_{d+|\bq|}^{(mn_0+1)}(1)\in HZ_{2(d+|\bq|)}(\bepsilon_{m,1})$.
The morphism $G_*^{(a+1,b+1)}(\bsigma) \to G_*^{(a+1,b+1)}((\bepsilon^2,
\bsigma))$, $\alpha\mapsto a_d\compii \alpha$, is an isomorphism,
let us write $\alpha\mapsto a_d^{-1}\compii \alpha$ its inverse morphism.
We define the morphism (\ref{iso:periodic}) by
$\alpha\mapsto a_d^{-1}\compii a_{d+|\bq|} \compii \alpha$.
The proof of \cite[Proposition~4.10]{OnHoferZehnderGF}
applies with these formal adaptations
(in the proof, the generator $a_{-1}$ of degree $-2$ must also be replaced by
the generator $a_{d-|\bq|}$ of degree $2(d-|\bq|)$).

\begin{prop}\label{prop:periodicity}
    The morphism (\ref{iso:periodic}) is an isomorphism
    commuting with inclusion and boundary morphisms.
\end{prop}

Following the proof in the unweighted case, one can now define
spectral invariants $c_k(\bsigma)$ in our generalized setting.

\begin{thm}\label{thm:spectral}
    Let $\bsigma$ be a tuple of small $\RS$-equivariant Hamiltonian
    diffeomorphisms
    associated with the Hamiltonian diffeomorphism $\varphi$ of $\CP(\bq)$.
    As a graded module over a coefficient ring $R$ whose characteristic is either $0$
    or prime with the weights,
    \begin{equation*}
        G_*^{(-\infty,+\infty)}(\bsigma) = \bigoplus_{k\in\Z} R\alpha_k
    \end{equation*}
    for some non-zero $\alpha_k$'s with $\deg\alpha_k = 2k$.
    For all $k\in\Z$, let
    \begin{equation*}
        c_k(\bsigma) := \inf \left\{ t\in \R\ |\ \alpha_k \in
            \im \left( G_*^{(-\infty,t)}(\bsigma) \to
        G_*^{(-\infty,+\infty)}(\bsigma) \right)\right\}.
    \end{equation*}
    Then for all $k\in\Z$, $c_k(\bsigma)\in\R$ is an action value
    of $\bsigma$ and
    $c_{k+|\bq|}(\bsigma)=c_k(\bsigma)+1$. Moreover \[c_k(\bsigma) \leq c_{k+1}(\bsigma)\] for all $k \in \Z,$ and if there exists $k\in\Z$ such that $c_k(\bsigma)=c_{k+1}(\bsigma)$,
    then $\varphi$ has infinitely many fixed points of action $c_k(\bsigma)$.
    If $d+1$ consecutive $c_k(\bsigma)$'s are equal then $\varphi=\id$.
\end{thm}

As a corollary, we get the generalization of the Fortune-Weinstein theorem stated in the
introduction of the article.

\begin{proof}[Proof of Theorem~\ref{thm:arnold}]
Let $\bsigma$ be a tuple associated with the Hamiltonian diffeomorphism
$\varphi\in\ham(\CP(\bq))$ that has finitely many fixed points.
The spectral values $c_k(\bsigma)$ must all be distinct.
Since $(c_k(\bsigma))$ is increasing and satisfies $c_{k+|\bq|}(\bsigma) = c_k(\bsigma) +1$,
there are exactly $|\bq|$ distinct spectral values inside $[0,1)$.
But the number of action values inside $[0,1)$ associated with a fixed point $z$
equals $\order(z)$, so the conclusion follows.
\end{proof}

The Poincaré duality in singular homology implies the following duality,
according to the proof of \cite[Proposition~4.13]{OnHoferZehnderGF}.

\begin{prop}\label{prop:poincare}
Let $\bsigma$ be a tuple of small $\RS$-equivariant Hamiltonian diffeomorphisms of
$\C^{d+1}$. There exists a duality isomorphism between generating functions homology
and cohomology
\begin{equation*}
PD : G^*_{(a,b)}(\bsigma) \xrightarrow{\sim} G_{2d-*}^{(-b,-a)}(\bsigma^{-1}),
\end{equation*}
with $-\infty\leq a\leq b\leq +\infty$ and $a,b$ not action values.
This isomorphism is natural: it commutes with inclusion and boundary maps.
\end{prop}

\begin{cor}\label{cor:spectralinverse}
    Let $\bsigma$ be a tuple of small $\RS$-equivariant
    Hamiltonian diffeomorphisms, then
    \begin{equation*}
        c_k (\bsigma^{-1}) = - c_{d-k} (\bsigma).
    \end{equation*}
\end{cor}

Following the proof in the unweighted case, the properties of the composition morphisms
imply the sub-additivity of the spectral invariants.

\begin{prop} Given any tuples $\bsigma$ and $\bsigma'$ of small
    $\RS$-equivariant Hamiltonian diffeomorphisms, one has
    \begin{equation*}
        c_{k+l-d}((\bsigma,\bepsilon,\bsigma')) \leq
        c_k(\bsigma) + c_l(\bsigma').
    \end{equation*}
\end{prop}

We can now associate to every $\bsigma$ a persistence module
$(G_*^{(-\infty,t)}(\bsigma))_t$ that satisfies
the ``periodicity'' property
$G_*^{(-\infty,t+1)}(\bsigma) \simeq G_{*+2|\bq|}^{(-\infty,t)}(\bsigma)$,
the isomorphism being an isomorphism of persistence module
according to the naturality of (\ref{iso:periodic}).
Let us refer to \cite[Section~3.1]{OnHoferZehnderGF} and references therein
for a quick review about persistence modules and barcodes
and let us just recall that persistence modules are $\R$-families
of vector spaces $(V^t)_{t\in\R}$ with natural maps
$V^t\to V^s$ for $t\leq s$ and that, under some finiteness
and continuity conditions (that are satisfied in our case),
one can associate a multiset $\bB(V^t)$ of intervals $[a,b)$
called the barcode of $(V^t)$
satisfying for all $t_0\in\R$,
\begin{equation*}
    \card \{ I \in \bB(V^t)\ |\ t_0\in I\} =
    \dim V^{t_0}.
\end{equation*}
While discussing barcodes properties of $(G_*^{(-\infty,t)}(\bsigma))_t$,
we assume that the persistence module is over a field
(which characteristic is either $0$ or prime with the weights)
and that the number of fixed points in $\CP(\bq)$ associated
with $\bsigma$ is finite.
Since this periodicity property shifts the degree
by a constant positive integer $2|\bq|$, it
induces a permutation of the bars of the barcode 
sending a bar $[a,b)$ on a bar $[a+1,b+1)$ that
generates a free $\Z$-action on the bars.
A family of representatives of the bars is given
by the union of the barcodes of
$(G_k^{(-\infty,t)}(\bsigma))_t$ for $0\leq k\leq 2|\bq| - 1$.
The infinite bars of the barcode are exactly the multiset of
intervals $I_k:=[c_k(\bsigma),+\infty)$, $k\in \Z$,
the positive generator of the $\Z$-action sending $I_k$ to $I_{k+|\bq|}$,
so that there are exactly $|\bq|$ $\Z$-orbits of infinite bars.
Still following the proof in the unweighted case, we get
the following interpretation of the homology count of fixed points $N(\bsigma;\F)$
in terms of count of bars.

\begin{prop}\label{prop:N}
    Given a tuple $\bsigma$ of $\RS$-equivariant Hamiltonian
    diffeomorphisms with a finite number of fixed $\RS$-orbits,
    for every field $\F$ whose characteristic is either $0$ or prime with the weights,
    \begin{equation*}
        N(\bsigma;\F) = |\bq| + 2 K(\bsigma;\F),
    \end{equation*}
    where $K(\bsigma;\F)$ is the number of $\Z$-orbits of finite bars of the
    persistence module
    of $\bsigma$ over the field $\F$.
    In other words, $N(\bsigma;\F)$ is the number of (finite) extremities of
    a set of representative bars.
\end{prop}

The crucial point being that $N(\bsigma;\F) > |\bq|$ means that there exist finite bars.
Let us denote $\beta(\bsigma;\F)\geq 0$ the maximal length of a finite
bar of the barcode and $\betatot(\bsigma;\F)\geq 0$ the sum of
the lengths of representatives of the $\Z$-orbits
of finite bars.
A final key property whose proof does not differ from the unweighted case is the
universal bound on the length of finite bars.

\begin{thm}\label{thm:betamax}
    For every tuple of small $\RS$-equivariant Hamiltonian diffeomorphism
    $\bsigma$ generating a Hamiltonian diffeomorphism
    of $\CP(\bq)$ with finitely many fixed points
    and every field $\F$ whose characteristic is either $0$ or
    prime with each weight $q_j$,
    the longest finite bar of its barcode is less than 1:
    \begin{equation*}
        \betamax(\bsigma;\F) < 1.
    \end{equation*}
\end{thm}

\section{Smith inequality}
\label{se:smith}

In this section, we show
the natural extension of the Smith-type inequality stated in
\cite[Corollary~6.3]{OnHoferZehnderGF}.
Although the proof is rather similar to the unweighted case,
some technical adjustments must be made.
The assumption that the prime number used does not divide any of the weights
is crucial here.

\subsection{$\Z/p\Z$-action of a $p$-iterated generating function}

Let us fix a prime number $p\geq 3$ that does not divide any of the weights
$q_j$'s.
Let us fix $t\in\R$ and study the generating function
of $\rho_\bq(-t)\Phi$ expressed
\begin{equation*}
    F_{\bsigma_{m,t}} (\mathbf{v}) := \sum_{k=1}^{n}
        f_k\left(\frac{v_k + v_{k+1}}{2}\right) +
    \frac{1}{2}\la v_k,iv_{k+1}\ra,
\end{equation*}
where $\mathbf{v}:=(v_1,\ldots,v_n)\in (\C^{d+1})^n$
and the $f_k : \C^{d+1} \to\R$ are $S^1$-invariant and positively 2-homogeneous.
Thus $F_{\bsigma_{m,t}^p}:(\C^{n(d+1)})^p\to\R$ is invariant under the action of $\Z/p\Z$
by cyclic permutation of coordinates generated by
\begin{equation*}
    (\mathbf{v}_1,\mathbf{v}_2,\ldots,\mathbf{v}_p)
    \mapsto
    (\mathbf{v}_p,\mathbf{v}_1,\ldots,\mathbf{v}_{p-1}),
\end{equation*}
(here $\bsigma^p_{m,t}$ means $(\bsigma_{m,t})^p$).
The induced $\widehat{F}_{\bsigma_{m,t}^p}:\CP(\bq^{np})\to\R$ is then
invariant under the $\Z/p\Z$-action by permutation of (weighted) homogeneous coordinates
induced by
\begin{equation*}
    [\mathbf{v}_1:\mathbf{v}_2:\cdots:\mathbf{v}_p]
    \mapsto
    [\mathbf{v}_p:\mathbf{v}_1:\cdots:\mathbf{v}_{p-1}].
\end{equation*}

\begin{lem}\label{lem:paction}
The fixed points $(\CP(\bq^{np}))^{\Z/p\Z}$
of the above action are the disjoint union $\bigsqcup_q P_q$
of the $p$ following projective subspaces of weights $\bq^n$:
\begin{equation*}
    P_r := \left\{ \big[\mathbf{v} : \rho_{\bq^n}(r/p) \mathbf{v} : \rho_{\bq^n}(2r/p) \mathbf{v} : \cdots
    : \rho_{\bq^n}((p-1)r/p) \mathbf{v} \big] \ |\ [\mathbf{v}]\in \CP(\bq^n) \right\},
\end{equation*}
for $r\in \Z/p\Z$.
\end{lem}

\begin{proof}
Let $(\mathbf{v}_1,\ldots,\mathbf{v}_p)\in \sphere{}(\bq^{np})$ be a point whose
projection is in $(\CP(\bq^{np}))^{\Z/p\Z}$.
There exists $t\in \R$ such that $\mathbf{v}_{j+1} = \rho_{\bq^n}(t) \mathbf{v}_j$ for all $j$
so that every $\mathbf{v}_j$ is nonzero and  $\rho_{\bq^n}(pt)$ fixes it.
Let $m\in\N$ be the order of the isotropy group
of $\mathbf{v}_1$ under the $S^1$-action, so that
$pt = \frac{k}{m}$ for some $k\in\N$.
Since $m$ divides one of the weights, it is prime with $p$ so there exists
$u,v\in\Z$ such that $um+vp=1$.
Thus $t = k(\frac{u}{p}+\frac{v}{m})$.
Since $\rho_{\bq^n}(kv/m)$ fixes $\mathbf{v}_1$, one can assume $t=ku/p$
and we deduce that $[\mathbf{v}_1:\cdots:\mathbf{v}_p] \in P_{ku}$.

The other inclusion is clear.
\end{proof}

\begin{prop}\label{prop:paction}
For $r\in\Z/p\Z$ and $t\in (-m,m)$, let $g_{r,t}$ be the restriction to $P_r$ of
$\widehat{F}_{\bsigma^p_{m,t}}$. Up to a shift in degree,
\begin{equation*}
H_*(\{ g_{r,b} \leq 0 \} , \{ g_{r,a} \leq 0 \} ) \simeq G_*^{(a+r/p,b+r/p)}(\bsigma),
\end{equation*}
when $-m < pa < pb < m$, with $a+r/p$ and $b+r/p$ not action values of $\bsigma$
as well as $pa$ and $pb$ not action values of $\bsigma^p$.
\end{prop}

\begin{proof}
Given a family $(h_t)$ of maps $X\to\R$, we will use the notation
\begin{equation*}
G(h_t) := H_*(\{ h_b \leq 0 \} , \{ h_a \leq 0 \} ),
\end{equation*}
so that we must show $G(g_{r,t})\simeq G(\widehat{F}_{\bsigma_{m,t+r/p}})$, up
to a shift in degree.

Using the fact that the $f_k$'s are $S^1$-invariant,
\begin{multline*}
    \frac{1}{p}
    F_{\bsigma^p_{m,t}} (\mathbf{v},\rho_{\bq^n}(r/p)\mathbf{v},\ldots,\rho_{\bq^n}(r(p-1)/p) \mathbf{v}) 
    = \\ \sum_{k=1}^{n-1} \left[
        f_k\left(\frac{v_k + v_{k+1}}{2}\right) +
    \frac{1}{2}\la v_k,iv_{k+1}\ra \right]
        + f_n\left(\frac{v_n + \rho_\bq(r/p) v_1}{2}\right) +
    \frac{1}{2}\la v_n,i\rho_\bq(r/p) v_1\ra.
\end{multline*}
We apply the linear change of variables $\mathbf{v}\mapsto\mathbf{u}$
given by $u_k := v_k + (-1)^k\frac{I-\rho_\bq (r/p)}{2} v_1$ so that
\begin{equation*}
    \left\{
        \begin{array}{c c c}
            u_1 + u_2 &=& v_1 + v_2, \\
            u_2 + u_3 &=& v_2 + v_3, \\
                      &\vdots& \\
            u_{n-1} + u_n &=& v_{n-1} + v_n, \\
            u_n + u_1 &=& v_n + \rho_\bq (r/p)v_1.
        \end{array}
    \right.
\end{equation*}
A direct computation gives
\begin{equation*}
    \sum_{k=1}^{n-1} 
    \la v_k,iv_{k+1}\ra  + \la v_n,i\rho_\bq(r/p) v_1\ra
    = \sum_{k=1}^n \la u_k,i u_{k+1} \ra - 2
    \sum_{j=0}^d \tan\left(\frac{r\pi q_j}{p}\right)| u_{1,j}|^2,
\end{equation*}
for all integer $r\in \Z/p\Z$,
so that
\begin{eqnarray*}
    F_{\bsigma^p_{m,t}} (\mathbf{v},\rho_{\bq^n}(r/p) \mathbf{v},\ldots,\rho_{\bq^n}(r(p-1)/p) \mathbf{v})
    &=& p\left[F_{\bsigma_{m,t}} (\mathbf{u}) - 
    \sum_{j=0}^d \tan\left(\frac{r\pi q_j}{p}\right)| u_{1,j}|^2
    \right] \\
    &=& p\left[F_{\bsigma_{m,t}} (\mathbf{u}) +
    F_{\delta_{r/p}}(u_1) 
    \right].
\end{eqnarray*}
The last bracket is the fiberwise sum of a generating function of
$\rho_{\bq}(-t)\Phi$ and the elementary generating function of
$\delta_{r/p} = \rho_\bq(-r/p)$, that we denote $F_{\bsigma_{m,t}} +
F_{\delta_{r/p}}$, evaluating at $\mathbf{u}$.
From this change of coordinates, we deduce the isomorphism
\begin{equation*}
G(g_{r,t}) \simeq G\left(\widehat{F_{\bsigma_{m,t}} + F_{\delta_{r/p}}}\right).
\end{equation*}
We recall that in this case, an $\RS$-orbit of critical points of
the fiberwise sum is in one-to-one correspondence with an $\RS$-orbit of fixed
points
of the composed diffeomorphism $\rho_{\bq}(-t-r/p)\Phi$
(see the paragraph surrounding Equation~(\ref{eq:fiberwise})).
Let us identify $r$ with its representative in $\{ 0, 1,\ldots, p-1\}$.
Contrary to the unweighted case, the path $s\mapsto F_{\delta_{sr/p}}$,
$s\in [0,1]$, is not well-defined in general and one must add auxiliary variables
to this generating function.
The generating function $F_{\bepsilon_{1,r/p}}$ is a quadratic form
generating the same Hamiltonian diffeomorphism as the elementary
generating function $F_{\delta_{r/p}}$.
According to Lemma~\ref{lem:unicityeqgf}, there exists a linear fiberwise
isomorphism $(x;\xi)\mapsto (x;L(x,\xi))$ such that
\begin{equation*}
F_{\bepsilon_{1,r/p}}(x;L(x,\xi)) = F_{\delta_{r/p}}(x) + R(\xi),
\end{equation*}
where $R$ is a non-degenerate quadratic form.
Applying the fiberwise isomorphism $(x;\eta,\xi)\to (x;\eta,L(x;\xi))$,
one gets
\begin{equation*}
G\left(\widehat{F_{\bsigma_{m,t}} + F_{\bepsilon_{1,r/p}}}\right)
\simeq
G\left(\widehat{F_{\bsigma_{m,t}} + (F_{\delta_{r/p}}\oplus R)} \right).
\end{equation*}
Since $0$ is a regular value of
$\widehat{F_{\bsigma_{m,t}} + (F_{\delta_{r/p}}\oplus R)}$ 
for $t\in \{ a,b\}$ by assumption, \cite[Proposition~B.1]{Giv90} implies
\begin{equation*}
G\left(\widehat{F_{\bsigma_{m,t}} + (F_{\delta_{r/p}}\oplus R)} \right) \simeq
G\left(\widehat{F_{\bsigma_{m,t}} + F_{\delta_{r/p}}}\right),
\end{equation*}
up to a shift in degree (the index of $R$).
Therefore, we can now replace $F_{\delta_{r/p}}$ by $F_{\bepsilon_{1,r/p}}$.

Let $(f_{s,t})$ be the family of well-defined maps
\begin{equation*}
    f_{s,t} := F_{\bsigma_{m,t+(1-s)r/p}} + F_{\bepsilon_{1,sr/p}}
    ,\quad s\in [0,1].
\end{equation*}
The function $f_{s,t}$ is the fiberwise sum of a generating function
of $\rho_\bq(-t-(1-s)r/p)\Phi$ and a generating function of
$\rho_\bq(-sr/p)$ so
$0$ is a regular value of $f_{s,t}$ if and only if
$\rho_\bq(-t-r/p)\Phi$ does not have any $\RS$-orbit of fixed points,
that is if and only if $t+r/p$ is not an action value of $\bsigma$.
According to \cite[Proposition~4.7]{OnHoferZehnderGF}, $G(f_{0,t})\simeq
G(f_{1,t})$ so that
\begin{equation*}
G(g_{r,t}) \simeq
G\left(\widehat{ F_{\bsigma_{m,t+r/p}} + F_{\bepsilon_{1,0}} }\right).
\end{equation*}
Since $F_{\bepsilon_{1,0}}$ is generating the identity as the zero map,
by the same argument as above, one can replace $F_{\bsigma_{m,t+r/p}} + F_{\bepsilon_{1,0}}$
in the last expression with $F_{\bsigma_{m,t+r/p}}$,
the conclusion follows.
\end{proof}

\subsection{Application of Smith inequality and computation of $\betatot$}
\label{se:betatot}

According to Smith inequality,
\begin{equation}\label{eq:smithineq}
    \dim H_*(X;\F_p) \geq \dim H_*(X^{\Z/p\Z};\F_p),
\end{equation}
where $X$ is locally compact space or pair such that $H_*(X;\F_p)$ is
finitely generated, a space
on which acts the group $\Z/p\Z$ (see for instance
\cite[Chapter IV, \S4.1]{Bor60}).
Here $\dim H_*$ means the total dimension $\sum_k \dim H_k$.

\begin{prop}\label{prop:smith}
    Given any tuple $\bsigma$ of small $\RS$-equivariant Hamiltonian diffeomorphisms,
    for every prime number $p$ that is prime with any of the weights and every $a\leq b$ such that
    $a+r/p$ and $b+r/p$ are not action values of $\bsigma$
    and $pa$ and $pb$ are not action values of $\bsigma^p$,
\begin{equation*}\label{eq:GSmith}
    \dim G_*^{(pa,pb)}(\bsigma^p;\F_p) \geq
    \sum_{(1-p)/2\leq r\leq (p-1)/2} \dim G_*^{(a+r/p,b+r/p)}(\bsigma;\F_p).
\end{equation*}
\end{prop}

We will only prove the case $p\geq 3$ in order to simplify the exposition.
In order to treat the case $p=2$,
one should modify the argument given in \cite[Section~6.4]{OnHoferZehnderGF}
with arguments similar to the ones used above.

\begin{proof}
    Let us take $m> \max (|a|,|b|)$.
    Now, we apply the Smith inequality (\ref{eq:smithineq}) to the couple
    \begin{equation*}
        X := \left(\left\{ \widehat{F}_{\bsigma_{m,b}^p} \leq 0 \right\},
        \left\{ \widehat{F}_{\bsigma_{m,a}^p} \leq 0 \right\} \right).
    \end{equation*}
    Similarly to the unweighted case, for some $i_0 \in\N$,
    \begin{equation}\label{iso:sigmap}
    H_{*+i_0}(X) =
        HZ_*(\bsigma^p_{m,b},\bsigma^p_{m,a})
        \simeq
        G_*^{(pa,pb)}(\bsigma^p,pm)
        \simeq G_*^{(pa,pb)}(\bsigma^p).
    \end{equation}
    According to Lemma~\ref{lem:paction},
    \begin{equation*}
        X^{\Z/p\Z} \simeq \bigsqcup_{(1-p)/2\leq r\leq (p-1)/2} 
        \left( \left\{ \widehat{F}_{\bsigma_{m,b}^p}|_{P_r} \leq 0\right\},
        \left\{ \widehat{F}_{\bsigma_{m,a}^p}|_{P_r}  \leq 0\right\}\right).
    \end{equation*}
    According to Proposition~\ref{prop:paction}, up to a shift in degree,
    \begin{equation*}
    H_*(X^{\Z/p\Z};\F_p) \simeq
    \bigoplus_{(1-p)/2\leq r\leq (p-1)/2} G_*^{(a+r/p,b+r/p)}(\bsigma ; \F_p).
    \end{equation*}
    Therefore, Smith inequality (\ref{eq:smithineq})
    together with (\ref{iso:sigmap}) bring the
    conclusion.
\end{proof}

Deducing the Smith-type inequality for $\betatot$ is now identical to
the unweighted case: one can express $\betatot$ as an integral and
then apply Proposition~\ref{prop:smith}.

\begin{prop}\label{prop:betatotint}
    Let $\bsigma$ be a tuple of small $\RS$-equivariant Hamiltonian
    diffeomorphisms with a finite
    number of associated fixed points in $\CP(\bq)$.
    For every $a\in\R$, every integer $n\in\N^*$ and every field $\F$ whose
    characteristic is
    either $0$ or prime to any of the weights $q_j$,
    \begin{equation*}
        \betatot(\bsigma;\F) = \frac{1}{2} \left(
        \int_0^1 \dim G_*^{(a+t,a+t+n)}(\bsigma;\F)
    \,\ud t - n|\bq| \right).
    \end{equation*}
\end{prop}

\begin{cor}
    \label{cor:betatotSmith}
    For every tuple of small $\RS$-equivariant Hamiltonian
    diffeomorphisms with a finite
    number of associated fixed points in $\CP(\bq)$,
    for every prime number $p$ that is prime with any of the weights,
    \begin{equation*}
        \betatot(\bsigma^p;\F_p) \geq p \betatot(\bsigma;\F_p).
    \end{equation*}
\end{cor}

Another easy consequence of the integral formula is the
following proposition (a direct extension of
{\cite[Proposition~6.4]{OnHoferZehnderGF}}).

\begin{prop}\label{prop:universalcoef}
    For every tuple of small $\RS$-equivariant Hamiltonian
    diffeomorphisms $\bsigma$ with a finite
    number of associated fixed points in $\CP(\bq)$,
    there exists an integer $N\in\N$ such that
    for all prime number $p\geq N$,
    \begin{equation*}
        \betatot(\bsigma;\F_p) = \betatot(\bsigma;\Q).
    \end{equation*}
\end{prop}

\section{Proof of the main theorem}
\label{se:proof}

The proof of Theorem~\ref{thm:main} is essentially
identical to the one given in \cite{OnHoferZehnderGF},
we reproduce it with slight modifications for the reader's
convenience.
The proof of Theorem~\ref{thm:arnold} was given
below Theorem~\ref{thm:spectral}.

\begin{proof}[Proof of Theorem~\ref{thm:main}]
    Let $\bsigma$ be any tuple of $\RS$-equivariant Hamiltonian diffeomorphisms
    associated with $\varphi$, so that $N(\bsigma;\F) = N(\varphi;\F)$.
    Let us denote by $K(\bsigma;\F)$ the number of $\Z$-orbits of
    finite bars of the barcode associated with $\bsigma$ over the field $\F$.
    According to the universal coefficient theorem,
    one can assume that $\F=\Q$ if $\F$ has characteristic 0
    and $\F=\F_p$ if it has characteristic $p\neq 0$.

    Let us assume that $\F=\Q$.
    According to Proposition~\ref{prop:N},
    $N(\bsigma;\Q)>|\bq|$ implies that $K(\bsigma;\Q)>0$
    so the maximal length of a finite bar $\betamax(\bsigma;\Q) > 0$.
    According to Corollary~\ref{cor:betatotSmith},
    for all prime number $p$ prime with any of the weights,
    \begin{equation*}
        K(\bsigma^p;\F_p)\betamax(\bsigma^p;\F_p) \geq
        \betatot(\bsigma^p;\F_p) \geq p\betatot(\bsigma;\F_p).
    \end{equation*}
    Thus, by Proposition~\ref{prop:universalcoef},
    for all sufficiently large prime $p$,
    \begin{equation*}
        K(\bsigma^p;\F_p)\betamax(\bsigma^p;\F_p) \geq
        p\betatot(\bsigma;\Q) \geq p\betamax(\bsigma;\Q),
    \end{equation*}
    that is to say that $K(\bsigma^p;\F_p)\betamax(\bsigma^p;\F_p)$
    grows at least linearly with prime numbers $p$.
    According to Theorem~\ref{thm:betamax},
    $\betamax(\bsigma^p;\F_p) \leq 1$ so
    $K(\bsigma^p;\F_p)$ must diverge to $+\infty$
    with prime numbers $p$ and so must $N(\bsigma^p;\F_p)$
    by Proposition~\ref{prop:N}.
    Let $z_1,\ldots, z_n\in\CP^d$ be the fixed points of $\varphi$.
    According to Corollary~\ref{cor:gromollmeyer}, there exists $B>0$
    such that $\dim \lochom_*(\bsigma^p;z_k,t_j(z_k);\F_p) < B$ for all $k$,
    $j$,
    and all prime $p$ that does not divide any of the weights.
    Let $A\in\N$ be such that for all prime $p\geq A$,
    $N(\bsigma^p;\F_p) > nq_0 q_1\cdots q_d B$.
    Then, for all prime $p\geq A$, there must be at least one
    fixed point of $\varphi^p$ that is not one of the $z_k$'s,
    that is there must be at least one $p$-periodic point
    that is not a fixed point.
    Hence, the conclusion for the case $\F$ of characteristic 0.

    Let us assume that $\F=\F_p$ for some prime number $p$.
    By contradiction, let us assume that $\varphi$ has only finitely
    many periodic points of period belonging to
    $\{ p^k \ |\ k\in\N \}$.
    According to Corollary~\ref{cor:betatotSmith},
    \begin{equation*}
        \betatot(\bsigma^{p^k};\F_p) \geq p^k\betatot(\bsigma;\F_p),
        \quad \forall k\in\N,
    \end{equation*}
    in particular, $N(\bsigma^{p^k};\F_p)>|\bq|$ for all $k\in\N$.
    Thus, by taking a sufficiently large $p^k$-iterate of $\varphi$,
    one can assume that every periodic point of $\varphi$ of period $p^k$
    for some $k$ is an admissible fixed point of $\varphi$
    (see above Proposition~\ref{prop:gromollmeyer} for the definition
    of an admissible fixed point).
    According to Proposition~\ref{prop:gromollmeyer}, it implies
    that $N(\bsigma^{p^k};\F_p) = N(\bsigma;\F_p)$ for all $k\in\N$.
    But Corollary~\ref{cor:betatotSmith} together with
    Proposition~\ref{prop:N} imply that the left-hand side of this equation
    must diverge to $+\infty$ as $k$ grows,
    a contradiction.
\end{proof}

\appendix

\section{Homology properties of orbibundles}
\label{se:orbibundle}

In this appendix, we show a version of the Thom isomorphism and the existence of
the Gysin morphism for orbibundles.

We refer to \cite[\S 3]{Sat57} for a precise definition of the orbifold structure
of a smooth orbibundle 
(orbibundles are called $V$-bundles there).
We will only use the underlying topological spaces at stake
and work with the following topological version of orbifold
containing the underlying topological maps of the smooth definition.
A continuous map $\pi:E\to B$ is an $F$-orbibundle,
$F$ being a topological space or pair, if $B$ is covered
by open sets $(V_\alpha)$ satisfying the following trivialization property.
For all $\alpha$, there exist a topological space $U_\alpha$,
a finite group $\Gamma_\alpha$ acting continuously on $U_\alpha$,
a continuous action of $\Gamma_\alpha$ on $U_\alpha\times F$ of the form
\begin{equation*}
    \gamma\cdot (x,y) := (\gamma\cdot x, f_\alpha (\gamma,x,y)),\quad
    \forall (x,y)\in U_\alpha\times F,
    \forall \gamma\in\Gamma_\alpha,
\end{equation*}
for some map $f_\alpha$,
there also exist $\Gamma_\alpha$-invariant continuous map
$U_\alpha \to V_\alpha$ and $U_\alpha\times F\to \pi^{-1}(V_\alpha)$
inducing the respective homeomorphisms
$\varphi_\alpha : U_\alpha/\Gamma_\alpha \to V_\alpha$
and $\chi_\alpha : (U_\alpha\times F)/\Gamma_\alpha \to \pi^{-1}(V_\alpha)$
and making the following diagram commute:
\begin{equation*}
    \begin{gathered}
        \xymatrix{
            \pi^{-1}(V_\alpha) \ar[d]^-{\pi}
            & (U_\alpha\times F)/\Gamma_\alpha \ar[l]^-{\chi_\alpha}_-{\simeq}
            \ar[d]^-{\mathrm{pr}_1}
            & U_\alpha \times F \ar[l] \ar[d]^-{\mathrm{pr}_1} \\
            V_\alpha
            & U_\alpha/\Gamma_\alpha \ar[l]^-{\varphi_\alpha}_-{\simeq}
            & U_\alpha \ar[l]
        }
    \end{gathered},
\end{equation*}
where the $\mathrm{pr}_1$'s denote the projection on the first factor
and the unlabeled maps are quotient maps.

Given a ring of coefficients $R$ and an $R$-oriented manifold
$F$, we say that an $F$-orbibundle $\pi:E\to B$ is $R$-oriented
if there are a preferred orientation of each fiber $\pi^{-1}(b)$
and covering $(V_\alpha)$ such that the $\chi_\alpha$'s
and the $\Gamma_\alpha$-actions
respect the orientation fiberwise.
Such a covering of $\pi$ is called an $R$-oriented covering.

\begin{thm}[Thom isomorphism]\label{thm:Thom}
    Let $\pi:E\to B$ be an $(\R^n,0)$-orbibundle that is
    $R$-oriented (as an $\R^n$-orbibundle).
    Let us assume that the order of the finite linear groups $\Gamma_\alpha$ of
    an $R$-oriented covering of $\pi$ are prime with the characteristic of
    $R$ or that the characteristic of $R$ is zero.
    Then there exists a natural class $\tau\in H^n(E,E_0;R)$ called the Thom
    class of $\pi$ such that the respective morphisms
    $\alpha\mapsto \pi_* (\alpha\frown\tau)$
    and $u\mapsto \tau\smile\pi^* u$ are isomorphisms
    \begin{equation*}
        H_*(E,E_0;R) \xrightarrow{\simeq} H_{*-n}(B;R)
        \ \text{ and }\
        H^*(B;R) \xrightarrow{\simeq} H^{*+n}(E,E_0;R),
    \end{equation*}
    where $E_0\subset E$ denote the
    total space of the associated $(\R^n\setminus 0)$-orbibundle.
\end{thm}

\begin{proof}
    Let us define the restriction of the Thom class $\tau_\alpha$
    to $(E,E_0)\cap \pi^{-1}(V_\alpha)$
    for any $\alpha$.
    The key point of this generalization is that
    the quotient maps associated with $U_\alpha$
    and $U_\alpha\times (\R^n,\R^n\setminus 0)$ induce isomorphism in cohomology
    (and homology) over $R$ according to
    Lemma~\ref{lem:Borel}:
    \begin{equation}\label{eq:ThomBorel}
        H^*((E,E_0)\cap \pi^{-1}(V_\alpha);R)
        \xrightarrow{\simeq} 
        H^*\left(U_\alpha\times (\R^n,\R^n\setminus 0);R\right)^{\Gamma_\alpha},
    \end{equation}
    indeed, the characteristic of $R$ is either $0$ or prime
    with the order of $\Gamma_\alpha$.
    Let $\tau_\alpha'\in H^n(U_\alpha\times (\R^n,\R^n\setminus 0))$
    be the Thom class of the trivial $(\R^n,0)$-bundle
    $U_\alpha\times (\R^n,0)\to U_\alpha$.
    Since the action of $\Gamma_\alpha$ on $U_\alpha\times\R^n$
    sends oriented fiber to oriented fiber, it preserves
    $\tau_\alpha'$ and the Thom isomorphism on this trivial bundle
    is $\Gamma_\alpha$-equivariant.
    Therefore, $\tau_\alpha$ can naturally be defined as the
    inverse image of $\tau'_\alpha$ under (\ref{eq:ThomBorel}).
    The proof now follows \emph{verbatim} the proof of the usual
    Thom isomorphism theorem given in \cite[\S 10]{MS74}.
\end{proof}

Similarly to $\sphere{n}$-bundle,
one can canonically
include any $\sphere{n}$-orbibundle
inside an $(\R^{n+1},0)$-orbibundle $E\to B$
and consider the long exact sequence of pair $(E,E_0)$
to obtain the following corollary.

\begin{cor}[The Gysin long exact sequence]\label{cor:Gysin}
    Let $\pi : E \to B$ be an $\sphere{n}$-orbibundle that
    is $R$-oriented.
    Let us assume that the order of the finite linear groups $\Gamma_\alpha$ of
    an $R$-oriented covering of $\pi$ are prime with the characteristic of
    $R$ or that the characteristic of $R$ is zero.    
    Then there exists a natural long exact sequence in homology
    called the Gysin long exact sequence
    \begin{equation*}
        \cdots \to H_{*+1}(E;R) \xrightarrow{\pi_*}
        H_{*+1}(B;R) \xrightarrow{\cdot\frown e}
        H_{*-n}(B;R) \xrightarrow{\pi^*}
        H_*(E;R) \to \cdots,
    \end{equation*}
    where $e\in H^{n+1}(B;R)$ is called the Euler class of the orbibundle,
    it is the pullback of the Thom class of the associated
    $(\R^{n+1},0)$-orbibundle,
    and the natural morphism $\pi^* : H_*(B;R)\to H_{*+n}(E;R)$
    is called the Gysin morphism.
\end{cor}

\bibliographystyle{amsplain}
\bibliography{biblio} 

\end{document}